\documentclass[11pt]{amsart}
\usepackage{amsfonts,amssymb}

\newcommand{\w}{\omega}
\newcommand{\pack}{\mathrm{pack}}
\newcommand{\Pack}{\mathrm{Pack}}

\newcommand{\e}{\varepsilon}
\newcommand{\IN}{{\mathbb{N}}}
\newcommand{\I}{\mathcal I}
\newcommand{\IZ}{\mathbb Z}
\newcommand{\A}{\mathcal A}
\newcommand{\rank}{\mathrm{rank}}

\newcommand{\pr}{\mathrm{pr}}

\newtheorem{lemma}{Lemma}
\newtheorem{claim}{Claim}

\newtheorem{theorem}{Theorem}
\newtheorem{corollary}{Corollary}

\title[Packing index of universally small sets in Polish groups]{Constructing universally small subsets of a given packing index in Polish groups}
\author{Taras Banakh and Nadya Lyaskovska}
\address{Instytut matematyki, Uniwersytet Jana Kochanowskiego, Kielce, Poland}
\address{Department of Mathematics, Ivan Franko National University of Lviv, Ukraine}
\email{t.o.banakh@gmail.com, lyaskovska@yahoo.com}

\begin{document}

\begin{abstract} A subset of a Polish space $X$ is called {\em universally small} if it belongs to each ccc $\sigma$-ideal with Borel base on $X$. Under CH in each uncountable Abelian Polish group $G$ we construct a universally small subset $A_0\subset G$ such that $|A_0\cap gA_0|=\mathfrak c$ for each $g\in G$. For each cardinal number $\kappa\in[5,\mathfrak c^+]$ the set $A_0$ contains a universally small subset $A$ of $G$ with sharp packing index $\pack^\sharp(A_\kappa)=\sup\{|\mathcal D|^+:\mathcal D\subset \{gA\}_{g\in G}$ is disjoint$\}$ equal to $\kappa$.
\end{abstract}
\maketitle
\baselineskip14pt

\section{Introduction}
This paper is motivated by a problem of Dikranjan and Protasov \cite{DP} who asked if the group of integers $\IZ$ contains a subset $A\subset \IZ$ such that the family of shifts $\{x+A\}_{x\in\IZ}$ contains a disjoint subfamily of arbitrarily large finite cardinality but does not contains an infinite disjoint subfamily. This problem can be reformulated in the language of packing indices $\pack(A)$ and $\pack^\sharp(A)$, defined for any subset $A$ of an group $G$ by the formulas:
$$\pack(A)=\sup\{|\mathcal D|:\mbox{$\mathcal D\subset\{gA\}_{g\in G}$ is a disjoint subfamily}\}$$
and
$$\pack^\sharp(A)=\sup\{|\mathcal D|^+:\mbox{$\mathcal D\subset\{gA\}_{g\in G}$ is a disjoint subfamily}\}.$$
So, actually Dikranjan and Protasov asked about the existence of a subset $A\subset \IZ$ with packing
index $\pack^\sharp(A)=\aleph_0$. This problem was answered affirmatively in \cite{BL1} and \cite{BL2}. Moreover, in  \cite{La} the second author proved that for any cardinal $\kappa$ with $2\le\kappa\le|G|^+$ and $\kappa\notin\{3,4\}$ in any Abelian group $G$ there is a subset $A\subset G$ with packing index $\pack^\sharp(A)=\kappa$. By Theorem 6.3 of \cite{BLR}, such a set $A$ can be found in any subset  $L\subset A$ with Packing index  $\Pack(L)=1$ where
$$\Pack(A)=\sup\{|\mathcal A|:\mathcal A\subset\{gA\}_{g\in G}\mbox{ \ is $|G|$-almost disjoint}\}.$$
A family $\A$ of sets is called {\em $\kappa$-almost disjoint} for a cardinal $\kappa$ if $|A\cap A'|<\kappa$ for any distinct sets $A,A'\in\A$. So, being $1$-almost disjoint is equivalent to being disjoint.

A subset $A\subset G$ with small packing index can be thought as large in a geometric sense because in this case the group $G$ does not contain many disjoint translation copies of $A$. It is natural to compare this largeness property with other largeness properties that have topological or measure-theoretic nature. It turns out that a set of a group $A$ can have small packing index (so can be large in geometric sense) and simultaneously be small in other senses. In \cite{BLR} it was proved that each uncountable Polish Abelian group $G$ contains a closed subset $A\subset G$ that has large packing index $\Pack(A)=1$ but is nowhere dense and Haar null in $G$. According to Theorem 16.3 \cite{PB}, under CH (the Continuum Hypothesis), each Polish group $G$ contains a subset $A$ with packing index $\pack(A)=1$, which is {\em universally null} in the sense that $A$ has measure zero with respect to any atomless Borel probability measure on $G$. In this paper we move further in this direction and prove that under CH each uncountable Abelian Polish group $G$ contains a subset $A\subset G$ with large packing index $\Pack(A)=1$, which is {\em universally small} in the sense that it belongs to any ccc Borel $\sigma$-ideal on $G$. This fact combined with Theorem 6.3 of \cite{BLR} allows us to construct universally small subsets of a given packing index in uncountable Polish Abelian groups.

Following Zakrzewski \cite{Zak2} we call a subset $A$ of a Polish space $X$ {\em universally small} if $A$ belongs to each ccc $\sigma$-ideal with Borel base on $X$. By an {\em ideal} on a set $X$ we understand a family $\mathcal I$ of subsets of $X$ such that
\begin{itemize}
\item $\cup\I=X\notin\I$;
\item for any sets $A,B\in\I$ we get $A\cup B\in\I$;
\item for any sets $A\in\I$ and $B\subset X$ we get $A\cap B\in\I$.
\end{itemize}
An ideal $\I$ on a Polish space $X$ is called
\begin{itemize}
\item a {\em $\sigma$-ideal} if $\cup\A\in\I$ for any countable subfamily $\A\subset\I$;
\item an {\em ideal with Borel base} if each set $A\in\I$ lies in a Borel subset $B\in\I$;
\item a {\em ccc ideal} if $X$ contains no uncountable disjoint family of Borel subsets outside $\I$.
\end{itemize}
Standard examples of ccc Borel $\sigma$-ideals are the ideal $\mathcal M$ of meager subsets of a Polish space $X$ and the ideal $\mathcal N$ of null  subsets with respect to an atomless Borel $\sigma$-additive measure on $X$. This implies that a universally small subset $A$ is universally null and universally meager. Following \cite{Zak1} we call a subset $A$ of a Polish space $X$ to be {\em universally meager} if for any Borel isomorphism $f:A\to 2^\w$ the image $f(A)$ is meager in the Cantor cube $2^\w$. Universally small sets were introduced by P.Zakrzewski \cite{Zak2} who constructed an uncountable universally small subset in each uncountable Polish space. It should be mentioned that there are models of ZFC \cite[\S5]{Miller} in which all universally small sets in Polish spaces have cardinality $\le\aleph_1<\mathfrak c$. In such models any universally small set $A$ in the real line has maximal possible packing index $\pack(A)=\Pack(A)=\mathfrak c$. This fact shows that the following theorem, which is the main result of this paper, necessarily has consistency nature and cannot be proved in ZFC.

\begin{theorem}\label{main} Under CH, each uncountable Abelian Polish group $G$ contains a universally small subset $A_0\subset G$ with Packing index $\Pack(A_0)=1$.
\end{theorem}

Combining this theorem with Theorem 6.4 of \cite{BLR} we get

\begin{corollary} Under CH, for any cardinal $\kappa\in[2,\mathfrak c^+]$ with $\kappa\notin\{3,4\}$ any uncountable Polish group $G$ contains a universally small subset $A$ with sharp packing index $\pack^\sharp(A)=\kappa$.
\end{corollary}

\section{Universally small sets from coanalytic ranks}

In this section we describe a (known) method of constructing universally small sets, based on coanalytic ranks. Let us recall that a subset $A$ of a Polish space $X$ is
\begin{itemize}
\item {\em analytic} if $A$ is the continuous image of a Polish space;
\item {\em coanalytic} if $X\setminus A$ is analytic.
\end{itemize}
By Souslin's  Theorem \cite[14.11]{Ke}, a subset of a Polish space is Borel if and only if it is analytic and coanalytic.

It is known \cite[34.4]{Ke} that each coanalytic subset $K$ of a Polish space $X$ admits a {\em rank function} $\rank:K\to\w_1$ that has the following properties:
\begin{enumerate}
\item for every countable ordinal $\alpha$ the set $B_\alpha=\{x\in K:\rank(x)\le\alpha\}$ is Borel in $X$;
\item each analytic subspace $A\subset K$ lies in some set $B_\alpha$, $\alpha<\w_1$.
\end{enumerate}

The following fact is known and belongs to mathematical folklore (cf. \cite[5.3]{Miller}). For the convenience of the reader we attach a short proof.

\begin{lemma}\label{l1} Let $K$ be a coanalytic non-analytic set in a Polish space $X$ and $\rank:K\to\w_1$ be a rank function. For any transfinite sequence of points $x_\alpha\in K\setminus B_\alpha$, $\alpha\in\w_1$, the set $\{x_\alpha\}_{\alpha\in\w_1}$ is universally small in $X$.
\end{lemma}

\begin{proof} Given any ccc Borel $\sigma$-ideal $\I$ on $X$, use the classical Szpilrajn-Marczewski Theorem \cite[\S11]{Ku} to conclude that the coanalytic set $K$ belongs to the completion $\mathcal B_\I(X)=\{A\subset X:\exists B\in \mathcal B(X)\;\;A\triangle B\in\I\}$ of the $\sigma$-algebra of Borel subsets of $X$ by the ideal $\mathcal I$. Consequently, there is a Borel subset $B\subset K$ of $X$ such that $K\setminus B\in\mathcal I$. By the property of the rank function, the Borel set $B$ lies in $B_\beta$ for some countable ordinal $\beta$. Then the set $\{x_\alpha\}_{\alpha<\w_1}$ belongs to the $\sigma$-ideal $\I$, being the union
of the countable set $\{x_\alpha\}_{\alpha\le \beta}$ and the set $\{x_\alpha\}_{\beta<\alpha<\w_1}\subset K\setminus B_\alpha\subset K\setminus B$ from the ideal $\I$.
\end{proof}

In order to prove Theorem~\ref{main} we shall combine Lemma~\ref{l1} with the following technical lemma that will be proved in Section~\ref{pf:l}.

\begin{lemma} \label{coan} For any uncountable Polish Abelian group $G$ there are a non-empty open set $U\subset G$ and a co-analytic subset $K$ of $G$ such that $U\subset(K\setminus A)-(K\setminus A)$ for
any analytic subspace $A\subset K$ of $G$.
\end{lemma}

\section{Proof of Theorem~\ref{main}} Assume the Continuum Hypothesis. Given an uncountable Polish Abelian group $G$ we need to construct a universally small subset $A\subset G$ with $\Pack(A)=1$. We shall use the additive notation for denoting the group operation on $G$. So, $0$ will denote the neutral element of $G$. For two subsets $A,B\subset G$ we put $A+B=\{a+b:a\in A,\;b\in B\}$ and $A-B=\{a-b:a\in A,\;b\in B\}$.

 By Lemma ~\ref{coan}, there are a non-empty open set $U\subset G$ and a
coanalytic subset $K$ such that $U\subset (K\setminus B)-(K\setminus B)$ for any Borel subset $B\subset K$ of $G$. This implies that the coanalytic set $K$ is not Borel in $G$. Let $\rank:K\to \w_1$ be a rank function for $K$. This function induces the decomposition $K=\bigcup_{\alpha<\w_1}B_\alpha$ into Borel sets $B_\alpha=\{x\in K:\rank(x)\le\alpha\}$, $\alpha<\w_1$, such that each  Borel subset $B\subset K$ of $G$ lies in some set $B_\alpha$, $\alpha<\w_1$.

The Continuum Hypothesis allows us to choose an enumeration $U=\{u_\alpha\}_{\alpha<\omega_1}$ of the open set $U$ such that for every $u\in U$ the set $\Omega_u=\{\alpha<\w_1: u_\alpha=u\}$ is uncountable. The separability of $G$ yields a countable subset $C\subset G$ such that $G=C+U$.

By induction, for every $\alpha<\w_1$ find two points $x_\alpha,y_\alpha\in K\setminus (B_\alpha\cup\{x_\beta:\beta<\alpha\})$ such that $x_\alpha-y_\alpha=u_\alpha$. Such a choice is always possible as $U\subset (K\setminus B)-(K\setminus B)$ for any Borel subset $B\subset K$ of  $G$. Lemma~\ref{l1} guarantees that the sets $\{x_\alpha\}_{\alpha<\w_1}$ and $\{y_\alpha\}_{\alpha<\w_1}$ are universally small in $G$ and so is the set $A=\{c+x_\alpha,y_\alpha:c\in C,\;\alpha<\w_1\}$. It remains to prove that $\Pack(A)=1$. This equality will follow as soon as we check that for every point $z\in G$ the intersection $A\cap (z+A)$ has cardinality of continuum. Since $C+U=G$, we can find elements $c\in C$ and $u\in U$ such that $z=c+u$. The choice of the enumeration $\{u_\alpha\}_{\alpha<\w_1}$  guarantees that the set $\Omega_u=\{\alpha<\w_1:u_\alpha=u\}$ has cardinality continuum. Now observe that for every $\alpha\in \Omega_u$ we get $z=c+u=c+u_\alpha=c+x_\alpha-y_\alpha$ and hence $c+x_\alpha=z+y_\alpha\in A\cap (z+A)$, which implies that the intersection $A\cap (z+A)\supset\{c+x_\alpha\}_{\alpha\in \Omega_u}$ has cardinality of continuum.

\section{Proof of Lemma~\ref{coan}}\label{pf:l}
 Fix an invariant metric $d\le 1$ generating the topology of $G$.
This metric is complete because the group $G$ is Polish. The  metric $d$ induces a norm $\|\cdot\|:G\to[0,1]$ on $G$ defined by $\|x\|=d(x,0)$. For an $\e>0$ by $B(\e)=\{x\in G:\|x\|<\e\}$ and $\bar B(\e)=\{x\in G:\|x\|\le\e\}$ we shall denote the open and closed $\e$-balls centered at zero.

We define a subset $D$ of $G$ to be {\em $\varepsilon$-separated}
if $d(x,y)\ge\varepsilon$ for any distinct points $x,y\in D$.
By Zorn's Lemma, each $\e$-separated subset $S$ of any subset $A\subset G$ can be enlarged to a
maximal $\e$-separated subset $\tilde S$ of $A$. This set $\tilde S$ is {\em $\e$-net} for $A$ in the sense that for each point $a\in A$ there is a point $s\in \tilde S$ with $d(a,s)<\e$.

Fix any non-zero element $a_{-1}\in G$ and let $\e_{-1}=\frac1{12}\|a_{-1}\|$.
By induction we can define a sequence $(\e_n)_{n\in\w}$ of positive real numbers and a sequence $(a_n)_{n\in\w}$ of points of the group $G$ such that
\begin{itemize}
\item $16\e_n\le \|a_n\|<\e_{n-1}$ for every $n\in\w$.
\end{itemize}
For every $n\in\w$, fix a maximal $2\e_n$-separated subset $X_n\ni 0$ in the ball $B(2\e_{n-1})$.

The choice of the  sequence $(\e_n)$ guarantees that the series $\sum_{n\in\w}\e_n$ is
convergent and thus for any sequence $(x_n)_{n\in\w}\in \prod_{n\in\w}X_n$ the
series $\sum_{n\in\w}x_n$ is convergent in $G$ (because
$\|x_n\|<2\e_{n-1}$ for all $n\in\IN$).
Therefore the following subsets of the group $G$ are well-defined:
$$
\begin{aligned}
\Sigma_0&=\big\{\sum_{n\in \omega}x_{2n}:(x_{2n})_{n\in\w}\in\prod_{n\in\w} X_{2n}\big\},\\
\Sigma_1&=\big\{\sum_{n\in \omega}x_{2n+1}:(x_{2n+1})_{n\in\w}\in\prod_{n\in\w}
X_{2n+1}\big\}.
\end{aligned}$$

These sets have the following properties:

\begin{claim}\label{cl1}
\begin{enumerate}
\item  $\Sigma_0\cup\Sigma_1\subset B(4\e_{-1})$.
\item $B(2\e_{-1})\subset \Sigma_1+\Sigma_0$.
\item For every $i\in\{0,1\}$ the closure $\overline{\Sigma_i-\Sigma_i}$ of the set $\Sigma_i-\Sigma_i$ in $G$ is not a neighborhood of zero.
\end{enumerate}
\end{claim}

\begin{proof}
1. For every point $x\in \Sigma_0\cup\Sigma_1$ we can find a sequence $(x_n)_{n\in\w}\in\prod_{n\in\w}X_n$ with $x=\sum_{n=0}^\infty x_n$ and observe that $$\|x\|\le\sum_{n=0}^\infty\|x_n\|\le\sum_{n=0}^\infty 2\e_{n-1}< \sum_{n\in\w}\frac{2\e_{-1}}{16^{n}}<4\e_{-1}.$$
\smallskip

2. Given any point $x\in B(2\e_{-1})$, find a point $x_0\in X_0$ such that $\|x-x_0\|<2\e_0$. Such a point $x_0$ exists as the set $X_0$ is a $2\e_0$-net in $B(2\e_{-1})$. Continuing by induction, for every $n\in\w$ find a point $x_n\in X_n$ such that  $\|x-\sum_{i=0}^nx_i\|<2\e_n$. After completing the inductive construction, we obtain a sequence $(x_n)_{n\in\w}\in\prod_{n\in\w}X_n$ such that  $$x=\sum_{n\in\w}x_n=\sum_{n\in\w}x_{2n}+\sum_{n\in\w}x_{2n+1}\in \Sigma_0+\Sigma_1.$$
\smallskip

3. We shall give a detail proof of the third statement for $i=0$ (for $i=1$ the proof is analogous).  Since the sequence $(a_{2k+1})_{k\in\w}$ converges to zero, it suffices to show that $d(a_{2k+1},\Sigma_0-\Sigma_0)>0$ for all $k\in\w$.

Given two points $x,y\in\Sigma_0$, we shall prove that $d(a_{2k+1},x-y)\ge \e_{2k+1}$. If $x=y$, then $d(a_{2k+1},x-y)=d(a_{2k+1},0)=\|a_{2k+1}\|> \e_{2k+1}$ by the choice of $a_{2k+1}$. So, we assume that $x\ne y$. Find infinite
sequences
$(x_{2n})_{n\in\w},(y_{2n})_{n\in\w}\in\prod_{n\in\w}X_{2n}$ with
$x=\sum_{n\in\w}x_{2n}$ and $y=\sum_{n\in\w}y_{2n}$.

Let $m=\min\{n\in\w:x_{2n}\ne y_{2n}\}$. If $m\ge k+1$, then
$$
\begin{aligned}
\|x-y\|=&\,\|\sum_{n\ge m}x_{2n}-y_{2n}\|\le\sum_{n\ge m}\|x_{2n}\|+\|y_{2n}\|\le\\
\le&\,2\sum_{n\ge m}2\e_{2n-1}\le
8\,\e_{2m-1}\le 8\,\e_{2k+1}<\|a_{2k+1}\|-\e_{2k+1}
\end{aligned}$$ and hence $d(x-y,a_{2k+1})\ge\e_{2k+1}$.

If $m\le k$, then
$$
\begin{aligned}
\|x-y\|=&\,\|(x_{2m}-y_{2m})+\sum_{n>m}(x_{2n}-y_{2n})\|\ge \|x_{2m}-y_{2m}\|-\sum_{n>m}(\|x_{2n}\|+\|y_{2n}\|)\ge\\
\ge&2\e_{2m}-2\sum_{n>m}2\e_{2n-1} \ge 2\e_{2m}-8\e_{2m+1}\ge \frac32\e_{2m}\ge \frac32\e_{2k}>\|a_{2k+1}\|+\frac12\e_{2k}
\end{aligned}$$ according to the choice of the point $a_{2k+1}$. Consequently, $d(x-y,a_{2k+1})\ge\frac12\e_{2k}\ge \e_{2k+1}$.
\end{proof}

A subset $C$ of $G$ will be called {\em a Cantor set} in $G$ if $C$ is homeomorphic to the Cantor cube $\{0,1\}^\w$. By the classical Brouwer's Theorem \cite[7.4]{Ke}, this happens if and only if $C$ is compact, zero-dimensional and has no isolated points.

\begin{claim}\label{cl2} For every $i\in\{0,1\}$ there is a Cantor set $C_i\subset B(\e_{0})$ such that the map $h_i:C_i\times \overline\Sigma_i\to G$, $h_i:(x,y)\mapsto x+y$, is a closed topological embedding.
\end{claim}

\begin{proof} Taking into account that $\overline{\overline\Sigma_i-\overline\Sigma_i}=\overline{\Sigma_i-\Sigma_i}$ is not a neighborhood of zero in $G$, and repeating the proof of Lemma~2.1 of \cite{BLR}, we can construct a Cantor set $C_i\subset B(\e_{0})$ such that for any distinct points $x,y\in C_i$ the shifts $x+\overline\Sigma_i$ and $y+\overline\Sigma_i$ are disjoint. This implies that the map $h_i:C_i\times\overline{\Sigma}_i\to G$, $h_i:(x,y)\mapsto x+y$, is injective. Taking into account that the set $C_i$ is compact and $\overline{\Sigma}_i$ is closed in $G$, one can check that the map $h_i$ is closed and hence a closed topological embedding.
\end{proof}

Observe that for every $i\in\{0,1\}$ the embedding $h_i$ has image $h_i(C_i\times\overline{\Sigma}_i)=C_i+\overline{\Sigma}_i\subset B(\e_{0})+\bar B(4\e_{-1})\subset
B(5\e_{-1})$. Now we modify the closed embeddings $h_0$ and $h_1$ to closed embeddings
$$\tilde h_0:C_0\times\overline{\Sigma}_0\to G,\;\;\tilde h_0:(x,y)\mapsto a_{-1}+x+y$$and
$$\tilde h_1:C_1\times \overline{\Sigma}_1\to G,\;\;\tilde h_1(x,y)=-x-y.$$ These embeddings have images
$\tilde h_0(C_0\times\overline{\Sigma}_0)\subset a_{-1}+B(5\e_{-1})$ and $\tilde h_1(C_1\times\overline{\Sigma}_1)\subset a_{-1}-B(5\e_{-1})=a_{-1}+B(5\e_{-1})$.
Since $\|a_{-1}\|=12\e_{-1}$, we conclude that the closed subsets $\tilde h_i(C_i\times \overline{\Sigma}_i)$, $i\in\{0,1\}$, of $G$ are disjoint.

For every $i\in\{0,1\}$ fix a coanalytic non-analytic subset $K_i$ in the Cantor set $C_i$. It follows that the disjoint union $K=\bigcup_{i=0}^1 \tilde h_i(K_i+\overline{\Sigma}_i)$ is a coanalytic subset of $G$.

The following claim completes the proof of the lemma and shows that the coanalytic set $K$ and the open set  $U=a_{-1}+B(\e_{-1})$ have the required property.

\begin{claim} $U\subset (K\setminus A)-(K\setminus A)$ for any analytic subspace $A\subset K$.
\end{claim}

\begin{proof} Given an analytic subspace $A\subset K$, for every $i\in\{0,1\}$, consider its preimage $A_i=\tilde h_1^{-1}(A)\subset C_i\times\overline{\Sigma}_i$ and its projection $\pr_i(A_i)$ onto the Cantor set $C_i$. It follows from $A\subset K$ and $\bigcap_{i=0}^1\tilde h_i(C_i\times\overline{\Sigma}_i)=\emptyset$ that each set $A_i$ is an analytic subspace of the coanalytic set $K_i$. Since the space $K_i$ is not analytic, there is a point  $c_i\in K_i\setminus \pr_i(A_i)$. It follows that
$$\bigcup_{i=0}^1\tilde h_i(\{c_i\}\times\Sigma_i)=(a_{-1}+c_0+\Sigma_0)\cup(-c_1-\Sigma_1)\subset K\setminus A$$ and hence
$$(K\setminus A)-(K\setminus A)\supset a_{-1}+c_0+\Sigma_0+c_1+\Sigma_1\supset a_{-1}+c_0+c_1+B(2\e_{-1})\supset a_{-1}+B(\e_{-1})=U$$according to Claim~\ref{cl1}(2).
The inclusion $B(\e_{-1})\subset c_0+c_1+B(2\e_{-1})$ follows from $c_0+c_1\in C_0+C_1\subset B(\e_0)+B(\e_0)\subset B(2\e_0)\subset B(\e_{-1})$.
\end{proof}

\end{document}